\documentclass[
12pt]{amsart}

\usepackage{amsmath}
\usepackage{amssymb}
\usepackage{latexsym}
\usepackage{amsfonts}
\usepackage{amsrefs}

\allowdisplaybreaks

\begin{document}
\newcommand{\eq}[2][label]{\begin{equation}\label{#1}#2\end{equation}}
\newcommand{\ci}[1]{_{ {}_{\scriptstyle #1}}}
\newcommand{\ti}[1]{_{\scriptstyle \text{\rm #1}}}

\newcommand{\norm}[1]{\ensuremath{\|#1\|}}
\newcommand{\abs}[1]{\ensuremath{\vert#1\vert}}
\newcommand{\nm}{\,\rule[-.6ex]{.13em}{2.3ex}\,}

\newcommand{\p}{\ensuremath{\partial}}
\newcommand{\pr}{\mathcal{P}}

\newcounter{vremennyj}

\newcommand\cond[1]{\setcounter{vremennyj}{\theenumi}\setcounter{enumi}{#1}\labelenumi\setcounter{enumi}{\thevremennyj}}

\newcommand{\pbar}{\ensuremath{\bar{\partial}}}
\newcommand{\db}{\overline\partial}
\newcommand{\D}{\mathbb{D}}
\newcommand{\R}{\mathbb{R}}
\newcommand{\T}{\mathbb{T}}
\newcommand{\C}{\mathbb{C}}
\newcommand{\N}{\mathbb{N}}
\newcommand{\bP}{\mathbb{P}}

\newcommand\cE{\mathcal{E}}
\newcommand\cP{\mathcal{P}}
\newcommand\cC{\mathcal{C}}
\newcommand\cH{\mathcal{H}}
\newcommand{\cU}{\mathcal{U}}

\newcommand{\be}{\mathbf{e}}

\newcommand{\la}{\lambda}

\newcommand{\td}{\widetilde\Delta}

\newcommand{\tto}{\!\!\to\!}
\newcommand{\wt}{\widetilde}
\newcommand{\shto}{\raisebox{.3ex}{$\scriptscriptstyle\rightarrow$}\!}

\newcommand{\La}{\langle }
\newcommand{\Ra}{\rangle }
\newcommand{\ran}{\operatorname{Ran}}
\newcommand{\tr}{\operatorname{tr}}
\newcommand{\codim}{\operatorname{codim}}
\newcommand\clos{\operatorname{clos}}
\newcommand{\spn}{\operatorname{span}}
\newcommand{\vf}{\varphi}
\newcommand{\f}[2]{\ensuremath{\frac{#1}{#2}}}


\newcommand{\entrylabel}[1]{\mbox{#1}\hfill}

\newenvironment{entry}
{\begin{list}{X}%
  {\renewcommand{\makelabel}{\entrylabel}%
      \setlength{\labelwidth}{55pt}%
      \setlength{\leftmargin}{\labelwidth}
      \addtolength{\leftmargin}{\labelsep}%
   }%
}%
{\end{list}}



\numberwithin{equation}{section}

\newtheorem{thm}{Theorem}[section]
\newtheorem*{thm*}{Theorem}
\newtheorem{lm}[thm]{Lemma}
\newtheorem{cor}[thm]{Corollary}
\newtheorem{prop}[thm]{Proposition}

\theoremstyle{remark}
\newtheorem{rem}[thm]{Remark}
\newtheorem*{rem*}{Remark}
\newtheorem*{tks*}{Acknowledgements}

\title[]{New estimates for the Beurling-Ahlfors operator on differential forms}

\author[S. Petermichl]{Stefanie Petermichl$^\dagger$}

\address{Stefanie Petermichl, D\'epartement de Math\'ematiques\\ Universit\'e de Bordeaux 1\\351, Cours de la Lib\'eration\\ 33405 Talence, France}
\email{stefanie@math.u-boreaux1.fr}
\thanks{$\dagger$ Research supported in part by a National Science Foundation DMS grant.}

\author[L. Slavin]{Leonid Slavin$^\ddagger$}

\address{Leonid Slavin, Department of Mathematics\\ University of Missouri, Columbia\\ Columbia, MO USA 65211}
\email{leonid@math.missouri.edu}
\thanks{$\ddagger$ Research supported in part by a National Science Foundation DMS grant.}

\author[B. D. Wick]{Brett D. Wick$^\ast$}

\address{Brett D. Wick, Department of Mathematics\\ University of South Carolina\\ Columbia, SC USA 29208}
\email{wick@math.sc.edu}
\thanks{$\ast$ Research supported in part by a National Science Foundation DMS grant.}


\date{\today}


\begin{abstract}
We establish new $p$-estimates for the norm of the generalized Beurling--Ahlfors transform $\mathcal{S}$ acting on form-valued functions. Namely, we prove that\linebreak $\norm{\mathcal{S}}_{L^p(\R^n;\Lambda)\to L^p(\R^n;\Lambda)}\leq n(p^{*}-1)$ where $p^*=\max\{p, p/(p-1)\},$ thus extending the recent Nazarov--Volberg estimates to higher dimensions. The even-dimensional case has important implications for quasiconformal mappings.  Some promising prospects for further improvement are discussed at the end.
\end{abstract}

\maketitle

\setcounter{tocdepth}{1}
\tableofcontents
\section*{Common notation}
\begin{tabular}{ll}
$:=$& equal by definition;\\
$E$& a Hilbert space, fixed throughout;\\
$\langle\cdot,\cdot\rangle_E$& the inner product in $E$;\\
 $\norm{\cdot}_E$& the norm on $E$ induced by $\langle\cdot,\cdot\rangle_E$;\\
$\abs{\cdot}$& either the absolute value or the Euclidean length in $\R^n$,\\
&clear from context;\\
$\partial_i$& the partial derivative with respect to $x_i,$ $\f{\partial}{\partial x_i}$;\\
$\partial_t$& the partial derivatives with respect to $t,$ $\f{\partial}{\partial t}$;\\
$C_c^\infty(\R^n;E)$& the space of infinitely differentiable compactly supported\\
&$E$-valued functions on $\R^n$;\\
$L^p(\R^n;E)$& the Lebesgue classes of $E$-valued functions on $\R^n$.
\end{tabular}
\section{Introduction}
\label{s1}
Singular integrals have historically played a very important role in the
theory of partial differential equations, for example in questions
regarding the higher integrability of weak solutions. In recent years the
importance of computing or estimating the the exact norms in $L^p$ of
singular integral operators has become clear.
In this light, it has been a long standing problem to find the exact
operator norm of the Beurling transform in scalar $L^p$.

The operator $\mathcal{T}:L^p(\C)\to L^p(\C),$ defined by
$$
\mathcal{T}(f)(z):=-\f{1}{\pi}\int_{\C}\frac{f(w)}{(w-z)^2}dA(w),
$$
is called the (planar) Beurling--Ahlfors transform.  This operator has the property that it turns $\bar z$-derivatives into $z$-derivatives:
$\mathcal{T}\left(\frac{\partial f}{\partial \bar z}\right) = \frac{\partial f}{\partial z}$.

The long standing conjecture of T. Iwaniec \cite{Iwan} is that, for $1<p<\infty,$ 
$$
\norm{\mathcal{T}}_{L^p(\mathbb{C})\to L^p(\mathbb{C})}=p^*-1,
$$
where
$$
p^*=\max\left\{p,\f{p}{p-1}\right\}.
$$

This problem has a long history, and has been reappearing in many papers on regularity and integrability of quasiregular maps.  Computing the exact norm of the operator acting on $L^p(\C)$ is important, since estimates of the norm yield information about the best integrability and minimal regularity of solutions to the Beltrami equation for example.  A more detailed description of the benefit of finding these sharp estimates can be found in \cite{PetermichlVol} and the references therein.

The operator $\mathcal{T}$ is a Calder\'on--Zygmund operator, and so is bounded on all $L^p(\C)$, but the standard methods of harmonic analysis fail to give good estimates of its norm. In search for new, more efficient tools, sharp Burkholder inequalities for martingales and the Bellman function technique have been shown to be very useful.
 
In 2000, F. Nazarov and A. Volberg, see \cite{VolNaz}, showed, using a special Bellman function, that $\norm{\mathcal{T}}_{p\to p}\le2(p^{*}-1)$.  Using martingale transforms, this result was duplicated by Ba\~nuelos and M\'endez-Hern\'andez \cite{BanMen}.  Recently, using a more refined version of these methods, Ba\~nuelos and Janakiraman \cite{BanPrabhu} were able to reduce the constant $2$ to approximately $1.575,$ the best known estimate to date.

There is a natural generalization of the planar Beurling--Ahlfors transform to differential forms on $\mathbb{R}^n.$ Just as the planar operator $\mathcal{T}$ on the complex plane determines many geometric and function-theoretic properties of quasiconformal mappings, estimates for the generalized Beurling--Ahlfors operator $\mathcal{S}$ on forms have applications in the theory of quasiregular mappings, where one computes the exponent of integrability, the distortion of the Hausdorff dimension under the images of these mappings, and the dimension of the removable singular set of the quasiregular map. These important applications are a reason the operator $\mathcal{S}$ has drawn continued interest.

T. Iwaniec and G. Martin studied $\mathcal{S}$ in \cite{IwanMart1}. They developed a complex `method of rotations' to deal with the even kernel of the operator and showed that 
$$
\norm{\mathcal{S}}_{L^p(\R^n;\Lambda)\to L^p(\R^n;\Lambda)}\le(n+1)
\norm{\mathcal{T}}_{L^p(\mathbb{C})\to L^p(\mathbb{C})}.
$$
They also conjectured that
\eq[e1]{
\norm{\mathcal{S}}_{L^p(\R^n;\Lambda)\to L^p(\R^n;\Lambda)}=(p^*-1).
}
It is well known that $\norm{\mathcal{S}}_{p\to p}\geq(p^*-1)$, but the estimate from above has proved elusive, just as it has in the planar case. Until now, the best known estimate was due to Ba\~nuelos and Lindeman \cite{BanLind}. Using Burkholder-type inequalities for martingale transforms, they showed that, under certain assumptions on $n,$ 
\eq[e2]{
\norm{\mathcal{S}}_{L^p(\R^n;\Lambda)\to L^p(\R^n;\Lambda)}\leq C(n)(p^*-1)
} 
with $C(n)=\frac43n-2.$

In this paper, we employ the line of reasoning developed by Petermichl and Volberg in a weighted setting \cite{PetermichlVol} and followed by Nazarov and Volberg in \cite{VolNaz}.  This line of investigation has been further explored by Dragi{\v{c}}evi{\'c} and Volberg in \cites{DragVol1,DragVol2,DragVol3}. This reasoning produces good $L^p\to L^p$ estimates for any operator that can be represented as a linear combination of products of two Riesz transforms. As such, the planar Beurling--Ahlfors transform serves as a prime object of study. The estimates are by duality: One recasts the action of the operator on two test functions in terms of derivatives of their heat extensions:
$$
\int_{\mathbb{R}^m}(R_kR_j\varphi)\psi=-2\iint_{\mathbb{R}^{m+1}_+}\partial_k\tilde{\varphi}\partial_j\tilde{\psi},
$$
where $\tilde{f}=\tilde{f}(x,t)$ is the heat extension of a function $f.$ This new integral is estimated using a Bellman function with appropriate size and concavity properties, together with a Green's function argument, which expresses the estimate in terms of the $L^p$ norms of the original test functions.

We are able to combine this approach with careful combinatorial considerations to obtain new estimates on $\norm{\mathcal{S}}_{p\to p}.$ Specifically, with no restriction on $n,$ we show that
$$
\norm{\mathcal{S}}_{L^p(\R^n;\Lambda)\to L^p(\R^n;\Lambda)}\le n(p^*-1).
$$
Although this result still falls far short from attaining the full Iwaniec--Martin conjecture, it improves significantly upon \eqref{e2}. In addition, we are hopeful that our approach has the potential to yield a qualitative reduction in the size of $C(n).$

The structure of the paper is as follows. In Section~\ref{s2}, we describe the basic objects under study, fix the notation to be used, and state the main results. In Section~\ref{s3}, contingent on an ``embedding-type'' theorem, we prove an upper estimate for the generalized Beurling--Ahlfors operator $\mathcal{S}$ in terms of derivatives of heat extensions of functions. The theorem, also proved in the section, relates the $L^p$ and $L^q$ norm of two test functions to integrals of derivatives of their heat extensions. The proof of this theorem, in turn, is based upon the existence of a certain Bellman function taking form valued variables, with delicate size and concavity properties. That existence is established in Section~\ref{s4} using Burkholder's famous estimates.
\begin{tks*}
We would like to express our sincere gratitude for the help and generosity of the organizers of the Linear Analysis and Probability Workshop at Texas A\&M University in the Summer 2006. This project began during  that workshop, when we were able to meet for a period of several weeks to talk about this problem. Special thanks are due to the organizers of the Thematic Program in Harmonic Analysis at the Fields Institute in Winter-Spring of 2008, as well as the staff of the Institute. It is during that semester that the paper took its final form. 
\end{tks*}
\section{Basic objects and main results}
\label{s2}
\subsection{Basic definitions}
\label{s2.1}
We now describe the principal objects under study.  We will be interested in differential forms on $\mathbb{R}^m$.  $\Lambda_k(\mathbb{R}^m)$ is the collection of $k$-forms in $\mathbb{R}^m$.  The dimension of this space is $\f{m!}{(m-k)!k!}$.  Let $\mathcal{I}_{m,k}$ denote the collection of ordered $k$ multi-indices $\overline{\imath}$, which are collections of $k$ distinct elements $i_1<\ldots<i_k$ of $\{1,\ldots,m\}$. There is a one-to-one correspondence between these multi-indices and subsets of $\{1,\ldots,m\}$ of order $k$. 

A basis for the $k$-forms is given as follows.  Let $\{e_1,\ldots, e_m\}$ denote the standard basis for $\mathbb{R}^m$.  For a multi-index $\overline{\imath}$, we set $e_{\overline{\imath}}:=e_{i_1}\wedge e_{i_2}\wedge\cdots\wedge e_{i_k}$.  Then $\Lambda_k$ is the complex vector space of $k$-forms on $\mathbb{R}^m$ and is given by
$$
\Lambda_k:=\textnormal{span}\{e_{\overline{\imath}}:\overline{\imath}\in\mathcal{I}_{m,k}\}.
$$
As a complex vector space $\Lambda_k$, inherits the usual Hermitian inner product structure.

Let $\Lambda:=\Lambda(\R^m)$ denote the graded algebra of differential forms in $\mathbb{R}^m$ constructed from the $\Lambda_k$.  The grading is with respect to the exterior product $\wedge$.  This space has an orthogonal decomposition 
$$
\Lambda:=\bigoplus_{k=0}^m\Lambda_k.
$$

We will also need the Hodge star operator $*:\Lambda\to\Lambda$.  The Hodge star is defined on basis elements and then extended via linearity to all of $\Lambda$.  Let $e_{\overline{\imath}}\in\Lambda_k$ and let $\overline{\imath}^c$ denote the complementary set of indices to $\overline{\imath}$.  Then $*e_{\overline{\imath}}=\sigma(\overline{\imath})e_{\overline{\imath}^c}$, where $\sigma(\overline{\imath})$ is the sign of the permutation which interchanges $\{1,2,\ldots, k\}$ with $\overline{\imath}$.  An alternative characterization of the Hodge star $*$ is given by the identity,
$$
\langle \alpha,\beta\rangle e_1\wedge\cdots\wedge e_n=\overline{\beta}\wedge*\alpha\quad\forall \beta\in \Lambda.
$$
One can easily see that $**\vert_{\Lambda_k}=(-1)^{k(m-k)}\textnormal{Id}$.

We will be interested in the Lebesgue classes of functions that take values in forms, $L^p(\R^m;\Lambda_k)$.  Let
$$
L^p(\R^m;\Lambda_k):=\left\{\sum_{\overline{\imath}\in\mathcal{I}_{m,k}}f_{\overline{\imath}}e_{\overline{\imath}}:f_{\overline{\imath}}\in L^p(\R^m)\right\};
$$
the norm on this space is
$$
\norm{f}_{L^p(\R^m;\Lambda_k)}^p:=\int_{\R^m}\left(\sum_{\overline{\imath}\in\mathcal{I}_{m,k}}\abs{f_{\overline{\imath}}(x)}^2\right)^{p/2}dx=\int_{\R^m}\norm{f(x)}_{\Lambda_k}^pdx.
$$

Consider functions in $C_c^\infty(\R^m;\Lambda_k)\cap L^2(\R^m;\Lambda_k)$.  The exterior derivative $\operatorname{d}$ defined via the rule
$$
\operatorname{d}f=\sum_{\overline{\imath}}\sum_{j=1}^m\partial_{j}f_{\overline{\imath}}e_{j}\wedge e_{\overline{\imath}}
$$
takes $k$-forms to $(k+1)$-forms.  Via the inner product on $L^2(\R^m;\Lambda_k)$
$$
\langle f,g\rangle_{L^2(\R^m;\Lambda_k)}:=\int_{\R^m}\langle f(x),g(x)\rangle_{\Lambda_k} dx=\int_{\R^m}\overline{f(x)}\wedge *g(x)dx
$$
we let $\delta$ denote the formal adjoint of $\operatorname{d}$.  Then the following statement simply becomes integration by parts
$$
\langle \operatorname{d}\phi,\psi\rangle=\langle \phi,\delta \psi\rangle
$$
which, in turn, yields $\delta = (-1)^{(m-k)k}*\operatorname{d}*$, and so $\delta$ takes $k$-forms to $(k-1)$-forms.

One then defines the generalized Beurling--Ahlfors operator as
$$
\mathcal{S}:=(\operatorname{d}\delta-\delta\operatorname{d})\Delta^{-1},
$$
and let $\mathcal{S}_k$ denote this operator restricted to $k$-forms.  Because of the orthogonal decomposition $\Lambda$ into the $\Lambda_k$, the structure of $\mathcal{S}$ is
$$
\mathcal{S}=
\left(
\begin{array}{ccccc}
\mathcal{S}_0 & 0 & 0 & 0 & 0\\
0 & \mathcal{S}_1 & 0 & 0 & 0\\
0 & 0 & \ddots &  0 & 0\\
0 & 0  & 0 & \mathcal{S}_{m-1} & 0\\
0 & 0  & 0 & 0 &\mathcal{S}_m
\end{array}
\right).
$$

The definition of $\mathcal{S}$ immediately shows that $*\mathcal{S}_k=-\mathcal{S}_{m-k}*$, and so the behavior of the Beurling--Ahlfors operator on $k$-forms is the same as on $(m-k)$-forms. 

We wish to view the generalized Beurling--Ahlfors operator as a Fourier multiplier.  Using the Fourier transform, one readily checks that
$$
\widehat{\operatorname{d}f}(\xi)=i\xi\wedge \hat{f}(\xi),
$$
where $\xi\in\R^m$ is viewed as an element in $\Lambda_1$.  Let $[\xi]$ denote the matrix representing the action of $\omega\to \xi\wedge\omega$ in $\Lambda$.  One sees that
$$
\widehat{\operatorname{d}f}(\xi)=[i\xi]\hat{f}(\xi).
$$

The symbol of the adjoint operator $\delta$ is also easily seen to $\widehat{\delta(f)}(\xi)=[i\xi]^t\hat{f}(\xi)$.  Therefore, as a Fourier multiplier, $\mathcal{S}_k$ is given by
$$
\widehat{\mathcal{S}_k}(f)(\xi)=\f{[\xi]^t[\xi]-[\xi][\xi]^t}{\abs{\xi}^2}\hat{f}(\xi):=M_\sharp(\xi)\hat{f}(\xi),
$$
where $M_\sharp$ is a $\f{m!}{(m-k)!k!}\times \f{m!}{(m-k)!k!}$ matrix.  In fact, $M_\sharp(x)$ can be viewed as the $k$-th exterior power of the Jacobian of inversion in the unit sphere in $\mathbb{R}^m$, i.e. $x\to \f{x}{\abs{x}^2}$.  The entries of the matrix are readily computed as
\begin{displaymath}
\widehat{(\mathcal{S}_k)}_{\overline{\imath},\overline{\jmath}}=
\left\{
\begin{array}{ccc}
-\sum_{p\in \overline{\imath}}\f{\xi_p^2}{\abs{\xi}^2}+\sum_{q\in\overline{\imath}^c}\f{\xi_q^2}{\abs{\xi}^2} & : & \overline{\imath}=\overline{\jmath}\\
-\f{2\xi_p\xi_q}{\abs{\xi}^2} & : & \overline{\imath}\setminus\overline{\jmath}=p\textnormal{ and }\overline{\jmath}\setminus\overline{\imath}=q\\
0 & : & \textnormal{otherwise}
\end{array}\right.
\end{displaymath}
where we order the basis by the lexicographic ordering on the multi-indices $\overline{\imath}$. Note that $\overline{\imath}\setminus\overline{\jmath}=p$ and $\overline{\jmath}\setminus\overline{\imath}=q$ means that the sets $\overline{\imath}$ and $\overline{\jmath}$ have the same cardinality and differ by exactly one element, which is $p$ or $q$, respectively. We include two examples for the convenience of the reader.  When $m=2$ and $k=1$, the matrix $\widehat{\mathcal{S}_1}$ is simply the usual Beurling operator.  When $m=4, k=2$ one can compute that
$$
\widehat{(\mathcal{S}_k)}_{\overline{\imath},\overline{\jmath}}=
\frac{1}{|\xi|^2}
\left(
\begin{array}{cccccc}
0&-2\xi_2\xi_3&-2\xi_2\xi_4&-2\xi_2\xi_3&-2\xi_1\xi_4&0\\
-2\xi_2\xi_3&0&-2\xi_3\xi_4&-2\xi_1\xi_2&0&-2\xi_1\xi_4\\
-2\xi_2\xi_4&-2\xi_3\xi_4&0&0&-2\xi_1\xi_2&-2\xi_1\xi_3\\
-2\xi_1\xi_3&-2\xi_1\xi_2&0&0&-2\xi_3\xi_4&-2\xi_2\xi_4\\
-2\xi_1\xi_4&0&-2\xi_1\xi_2&-2\xi_3\xi_4&0&-2\xi_2\xi_3\\
0&-2\xi_1\xi_4&-2\xi_1\xi_3&-2\xi_2\xi_4&-2\xi_2\xi_3&0
\end{array}
\right)+\frac{D}{|\xi|^2}
$$
where $D$ is the diagonal matrix 
$$
\frac{1}{|\xi|^2}
\left(
\begin{array}{cccccc}
A &0&0&0&0&0\\
0&B&0&0&0&0\\
0&0&C&0&0&0\\
0&0&0&-C&0&0\\
0&0&0&0&-B&0\\
0&0&0&0&0&-A
\end{array}
\right)
$$
with $A=-\xi_1^2-\xi_2^2+\xi_3^2+\xi_4^2$, $B=-\xi_1^2+\xi_2^2-\xi_3^2+\xi_4^2$ and $C=-\xi_1^2+\xi_2^2+\xi_3^2-\xi_4^2$.


The generalized Beurling--Ahlfors operator can also be built out of Riesz transforms.  The Riesz transforms are operators acting on $L^p(\R^m)$ given by
$$
R_l(f)(x):=\f{\Gamma(\f{m+1}{2})}{\pi^{\f{m+1}{2}}}\int_{\R^m}\f{x_l-y_l}{\abs{x-y}^{m+1}}f(y)dy,\quad l=1,\ldots, m.
$$
On the frequency side this operator takes the very compact form
$$
\widehat{R_l(f)}(\xi)=i\f{\xi_l}{\abs{\xi}}\hat{f}(\xi)\quad l=1,\ldots, m.
$$

The structure of the symbol of $\mathcal{S}_k$ then gives
\begin{displaymath}
\label{Sform}
\tag{S}
(\mathcal{S}_k)_{\overline{\imath},\overline{\jmath}}=
\left\{
\begin{array}{ccc}
\sum_{p\in \overline{\imath}}R_p^2-\sum_{q\in\overline{\imath}^c}R_q^2 & : & \overline{\imath}=\overline{\jmath}\\
2R_pR_q & : & \overline{\imath}\setminus\overline{\jmath}=p\textnormal{ and }\overline{\jmath}\setminus\overline{\imath}=q\\
0 & : & \textnormal{otherwise}.
\end{array}\right.
\end{displaymath}

Since $\mathcal{S}_k$ is a Fourier multiplier operator, it also has a description as a convolution-type singular integral operator.  In this form it is given by
\begin{equation}
\mathcal{S}_k(f)(x):=\left(1-\f{2k}{m}\right)f(x)-\f{\Gamma(\f{m+2}{2})}{\pi^{\f{m}{2}}}\int_{\R^m}\f{\Omega(x-y)}{\abs{x-y}^m}f(y)dy
\end{equation}
where $\Omega(x)=M_\sharp(x)+\left(\f{2k}{m}-1\right)\textnormal{Id}:\Lambda_k\to\Lambda_k$.  The matrix $\Omega$ has entries that are degree $2$ homogeneous harmonic polynomials of Riesz transforms.

The case of even dimensions is especially interesting, since the structure is very reminiscent of the classical Beurling operator.  When $m=2n$, the operator $\mathcal{S}_n$ has the form
$$
\mathcal{S}_n(f)(x):=-\f{n!}{\pi^n}\int_{\R^{2n}}\f{M_\sharp(x-y)}{\abs{x-y}^{2n}}f(y)dy.
$$

Let $\norm{\mathcal{S}}_{p\to p}$ denote the norm of the operator acting on $L^p(\R^m;\Lambda)$, and let $\norm{\mathcal{S}_k}_{p\to p}$ denote the norm its restriction to $k$-forms, i.e. the norm of the operator acting on $L^p(\R^m;\Lambda_k)$.
Because of the orthogonal decomposition of $\Lambda$, we have
$$
\norm{\mathcal{S}}_{p\to p}=\max_{0\leq k\leq n}\norm{\mathcal{S}_k}_{p\to p}.
$$
\subsection{Statement of main results}
\label{s2.2}
The main result of this paper is the following estimate on the norm of the Beurling--Ahlfors transform.  Namely we show the following
\begin{thm}
\label{MainResultGen}
Let $\mathcal{S}$ denote the Beurling--Ahlfors transform in $L^p(\R^{m};\Lambda)$.  Then
$$
\norm{\mathcal{S}}_{L^p(\R^{m};\Lambda)\to L^p(\R^{m};\Lambda)}\leq m(p^{*}-1).
$$
\end{thm}

The case of even dimensions is so interesting from the point of view of quasiregular mappings, as to warrant a separate statement. Namely, we trivially have the following
\begin{cor}
\label{MainResult}
Let $\mathcal{S}_{n}$ denote the Beurling--Ahlfors transform acting on $n$-forms in $\R^{2n}$.  Then
$$
\norm{\mathcal{S}_{n}}_{L^p(\R^{2n};\Lambda_n)\to L^p(\R^{2n};\Lambda_n)}\leq 2n(p^{*}-1).
$$
\end{cor}
We prove Theorem \ref{MainResultGen} for $p\geq 2$, with the result for $p\leq 2$ following by duality. Using a specially selected Bellman function, we will demonstrate a certain embedding theorem which provides a bilinear estimate of  derivatives of  heat extensions of functions by means of their $L^p$ norms.  Key to this approach is the following result, proved in Section~\ref{s4}.  
\begin{thm}
\label{MainEst}
Let $\varphi,\psi\in C_{c}^\infty(\R^{m};E)$ with $E$ a Hilbert space and $p\geq 2$.  Suppose that $\f{1}{p}+\f{1}{q}=1$ then
\begin{eqnarray*}
 & 2\iint_{\R^{m+1}_+}\left(\sum_{i=1}^m\norm{\partial_i\tilde{\varphi}(x,t)}^2_E\right)^{1/2}\left(\sum_{i=1}^m\norm{\partial_i\tilde{\psi}(x,t)}^2_E\right)^{1/2}dxdt\leq & \\
 & (p-1)\norm{\varphi}_{L^p(\R^m;E)}\norm{\psi}_{L^q(\R^m;E)}. &
\end{eqnarray*}
\end{thm}
We will apply this theorem with $E=\Lambda$ or $\Lambda_k$.  An immediate, and important for us, corollary is the following
\begin{cor}
\label{MainCor}
Let $\varphi,\psi\in C_{c}^\infty(\R^{m};E)$ with $E$ a Hilbert space and $p\geq 2$.  Suppose that $\f{1}{p}+\f{1}{q}=1$ then
\begin{eqnarray*}
 & 2\sum_{i,j=1}^{m}\iint_{\R^{m+1}_+}\norm{\partial_i\tilde{\varphi}(x,t)}_E\norm{\partial_j\tilde{\psi}(x,t)}_Edxdt\leq &\\
 & m(p-1)\norm{\varphi}_{L^p(\R^{m};E)}\norm{\psi}_{L^q(\R^{m};E)}. &
\end{eqnarray*}
\end{cor}
This can be seen by comparing $l^1$ and $l^2$ norms on $\R^m$.  Variants of these results can be found in \cites{DragTreilVol, DragVol1, DragVol2, DragVol3}.
\section{Estimates of the generalized Beurling--Ahlfors operator}
\label{s3}
In this section, we prove Theorem \ref{MainResultGen} given that Corollary \ref{MainCor} holds.  Given the block structure of $\mathcal{S}$ on $L^p(\R^{m};\Lambda)$, we will estimate each of the quantities $\norm{\mathcal{S}_k}_{p\to p}$ separately.  
\begin{prop}
\label{MainProp}
Let $\mathcal{S}_k$ denote the Beurling--Ahlfors transform in $L^p(\R^{m};\Lambda_k)$.  Then
$$
\norm{\mathcal{S}_k}_{L^p(\R^{m};\Lambda_k)\to L^p(\R^{m};\Lambda_k)}\leq m(p^{*}-1).
$$

\end{prop}


We will prove Proposition \ref{MainProp} in Section \ref{s3.2}, using duality. In matrix notation, we have
$$
\langle\mathcal{S}_k\varphi,\psi\rangle = \sum_{\overline{\imath},\overline{\jmath}}\int_{\R^{m}}\mathcal{S}_{\overline{\imath},\overline{\jmath}}(\varphi_{\overline{\jmath}})\psi_{\overline{\imath}}dx
$$
Since we want to use Corollary \ref{MainCor}, it becomes important to understand how the action of Riesz transforms interacts with the heat extension of a function.
\subsection{A Littlewood--Paley identity for heat extensions}
\label{s3.1}
The following proposition will be key for us.  It relates the action of double Riesz transforms on functions to the derivatives of their heat extensions.  Let $\varphi, \psi\in C_c^\infty(\R^{m};E)$.  Let $\tilde{f}$ denote the heat extension of $f$, which is given by the following formula
$$
\tilde{f}(y,t):=\f{1}{(4\pi t)^{m/2}}\int_{\R^{m}}f(x)\exp\left(-\f{\abs{x-y}^2}{4t}\right)dx,\quad(y,t)\in \R^{m+1}_+.
$$
\begin{lm}
\label{HeatExt}
Let $\varphi,\psi\in C_{c}^\infty(\R^{m};E)$.  Let $R_l$, $l=1,\ldots, m$, denote the Riesz transforms.  Then the integral
$$
\iint_{\R^{m+1}_+} \langle\partial_j\tilde{\varphi},\partial_k\tilde{\psi}\rangle_E dxdt
$$
converges absolutely and
$$
\int_{\R^{m}}\langle R_jR_k\varphi,\psi\rangle_E dx = -2\iint_{\R^{m+1}_+} \langle\partial_j\tilde{\varphi},\partial_k\tilde{\psi}\rangle_E dxdt
$$
\end{lm}
\begin{proof}
The proof of this lemma is just a component-wise application of the scalar formula that first appeared in \cite{PetermichlVol} with $E=\R^1$. It is based on the fact that a function is the integral of its derivative and Parseval's identity.  The heat kernel comes into play because $2\int_0^{\infty}e^{-2t|\xi|^2}dt=|\xi|^{-2}$.

\end{proof}
\subsection{Estimating the Beurling--Ahlfors transform via heat extensions}
\label{s3.2}
\begin{proof}[Proof of Lemma \ref{MainProp} and Theorem \ref{MainResultGen}]
We now will use Lemma \ref{HeatExt} in conjunction with the structure of the operator $\mathcal{S}_k$.  We split the action of the operator $\mathcal{S}_k$ into the diagonal and off diagonal parts. Namely,
\begin{eqnarray*}
\langle\mathcal{S}_k\varphi,\psi\rangle  & =  & \sum_{\overline{\imath},\overline{\jmath}}\int_{\R^{m}}\mathcal{S}_{\overline{\imath},\overline{\jmath}}(\varphi_{\overline{\jmath}})\psi_{\overline{\imath}}dx\\
 & = & \sum_{\overline{\imath}}\int_{\R^{m}}\mathcal{S}_{\overline{\imath},\overline{\imath}}(\varphi_{\overline{\imath}})\psi_{\overline{\imath}}dx+\sum_{\overline{\imath}\neq\overline{\jmath}}\int_{\R^{m}}\mathcal{S}_{\overline{\imath},\overline{\jmath}}(\varphi_{\overline{\jmath}})\psi_{\overline{\imath}}dx\\
 & =: & I+II.
\end{eqnarray*}

The estimate of $I$ is rather straightforward.  Using Lemma \ref{HeatExt} and (\ref{Sform}) we find that
$$
I=2\sum_{\overline{\imath}}\iint_{\R^{m+1}_+}\left(\sum_{i\in\overline{\imath}^c}\langle\partial_i\tilde{\varphi}_{\overline{\imath}},\partial_i\tilde{\psi}_{\overline{\imath}}\rangle-\sum_{i\in\overline{\imath}}\langle\partial_i\tilde{\varphi}_{\overline{\imath}},\partial_i\tilde{\psi}_{\overline{\imath}}\rangle\right)dxdt
$$

One notes that $\overline{\imath}\cup\overline{\imath}^c=\{1,\ldots, m\}$.  An application of H\"older's inequality implies thus 
\begin{eqnarray*}
\abs{I} & \leq & 2\iint_{\R^{m+1}_+}\sum_{\overline{\imath}}\sum_{i=1}^{m}\abs{\partial_i\tilde{\varphi}_{\overline{\imath}}}\abs{\partial_i\tilde{\psi}_{\overline{\imath}}}dxdt\\
 & \leq & 2\sum_{i=1}^{m}\iint_{\R^{m+1}_+}\norm{\partial_i\tilde{\varphi}}_{\Lambda_k}\norm{\partial_i\tilde{\psi}}_{\Lambda_k}dxdt.
\end{eqnarray*}

We now turn to estimating $II$.  Again the key tools will be Lemma \ref{HeatExt} and (\ref{Sform}).  
Simple combinatorial arguments show that the matrix which represents the operator $\mathcal{S}_k$ has the following properties when excluding the diagonal:  Each row has $k(m-k)$ non-trivial entries and the matrix which represents $\mathcal{S}_k$ has $\f{m!}{(m-k)!k!}$ rows/columns.  Let $\overline{\imath}$ denote one of the $\f{m!}{(m-k)!k!}$ possible multi-indices that generate the rows/columns of $\mathcal{S}_k$.  In this row, the term $R_l$, $l=1,\ldots, m$, appears k times if $l\in\overline{\imath}^c$ and $m-k$ times if $l\in\overline{\imath}$.  Each non-trivial off-diagonal term is of the form $2R_jR_k$, which we will write as $R_jR_k+R_jR_k$.  Thus we can recognize $2R_jR_k(f)g$ as $\partial_j\tilde{f}\partial_k\tilde{g}+\partial_k\tilde{f}\partial_j\tilde{g}$, introducing a factor of 2. All this implies that the number of non-zero terms in $II$ is
$$
2\f{m!}{(m-k)!k!}(m-k)k=m(m-1)\cdot2\f{(m-2)!}{(m-k-1)!(k-1)!}.
$$


To estimate $II$, one expands the expression 
$$
\sum_{\overline{\imath}\neq\overline{\jmath}}\mathcal{S}_{\overline{\imath},\overline{\jmath}}(\varphi_{\overline{\jmath}})\psi_{\overline{\imath}}
$$
and uses Lemma \ref{HeatExt} to remove the Riesz transforms and replace them by appropriate derivatives of the heat extensions of the components of the vectors $\varphi$ and $\psi$.  Finally, one divides all these terms into $m(m-1)$ classes, with each class containing $2\f{(m-2)!}{(m-k-1)!(k-1)!}$ elements.  The $m(m-1)$ possible classes correspond to the possibility of $\partial_i$ landing on $\tilde{\varphi}$ and $\partial_j$ landing on $\tilde{\psi}$.  For each class we then apply H\"older's inequality to conclude that
$$
\abs{II}\leq 2\sum_{i\neq j}^{m}\iint_{\R^{m+1}_+}\norm{\partial_i\tilde{\phi}}_{\Lambda_k}\norm{\partial_j\tilde{\psi}}_{\Lambda_k}dxdt.
$$

This estimate, along with $I$ yields 
\begin{eqnarray*}
\abs{\langle \mathcal{S}_k\varphi,\psi\rangle} & \leq & \abs{I}+\abs{II}\\
& \leq & 2\sum_{i,j=1}^{m}\iint_{\R^{m+1}_+}\norm{\partial_i\tilde{\varphi}}_{\Lambda_k}\norm{\partial_j\tilde{\psi}}_{\Lambda_k}dxdt.
\end{eqnarray*}

We now apply Corollary \ref{MainCor} and get
\begin{eqnarray*}
\abs{\langle \mathcal{S}_k\varphi,\psi\rangle}
& \leq & m(p-1)\norm{\varphi}_{L^p(\R^{m};\Lambda_k)}\norm{\psi}_{L^q(\R^{m};\Lambda_k)}.
\end{eqnarray*}

The last inequality follows from Theorem \ref{MainEst}.  Density and duality give Proposition \ref{MainProp}.  Since this argument is independent of $k$, upon taking that maximum over $0\leq k\leq m$ we see the validity of Theorem \ref{MainResultGen}.  
\end{proof}






\section{The Bellman function proof of Theorem \ref{MainEst}}
\label{s4}

The main results in this paper have been contingent upon Theorem \ref{MainEst}, so we now turn to its proof.
%
%
The following result is the key ingredient to arrive at Theorem \ref{MainEst}.
\begin{thm}
\label{Bellman}
For any fixed $p\geq 2$, its conjugate exponent $q,$ and Hilbert space $E$ define the domain
$$
D:=\{(\Xi,\Gamma,\xi,\gamma)\in \R\times\R\times E\times E: \norm{\xi}_E^p<\Xi,\norm{\gamma}_E^q<\Gamma\}.
$$
Let $K$ be any compact subset of $D$ and $\epsilon>0$.  There exists a function 
$B:=B_{\epsilon,K}(\Xi,\Gamma,\xi,\gamma)$, infinitely differentiable in a small neighborhood of $K$, such that
\begin{itemize}
\item[(i)] $0\leq B\leq (1+\epsilon)(p-1)\Xi^{1/p}\Gamma^{1/q}$,
\item[(ii)] $-d^2B\geq 2\norm{d\xi}_E\norm{d\gamma}_E$.
\end{itemize}
\end{thm}

In (ii), the inequality is in the sense of bilinear operators.  $-d^2B$ is a square matrix of size $2+2\dim{E}$ taking the argument $(d\Xi, d\Gamma, d\xi, d\gamma)$. Thus the operator on the right hand side is the one that returns $2\norm{d\xi}_E\norm{d\gamma}_E$.

We postpone the proof of Theorem \ref{Bellman}, but first use it to prove Theorem~\ref{MainEst}.
\begin{proof}[Proof of Theorem~\ref{MainEst}]
Let $\varphi,\psi\in C_c^\infty(\R^{m};E)$ fixed. Take $B$ from Theorem \ref{Bellman}, where we still have to choose the compact set $K$ and $\epsilon$.  We set
$$
b(x,t):=B(\widetilde{\norm{\varphi}_E^p}(x,t),\widetilde{\norm{\psi}_E^p}(x,t),\widetilde{\varphi}(x,t),\widetilde{\psi}(x,t)).
$$
It is important to note the order of the operations: $\widetilde{\norm{\varphi}_E^p}(x,t)$ means the heat extension of the function $\norm{\varphi}_E^p$. 
H\"older's inequality implies that the vector 
$$
v(x,t):=(\widetilde{\norm{\varphi}_E^p}(x,t),\widetilde{\norm{\psi}_E^p}(x,t),\widetilde{\varphi}(x,t),\widetilde{\psi}(x,t))\in D\quad\forall (x,t)\in\R^{m+1}_+.
$$
To choose $K$, we fix a compact $M$ of $\R^{m+1}_+$ and let $K$ be its  image under the map  $(x,t)\to v(x,t)$.  $K$ is compact as well because $\phi$ and $\psi$ belong to $C_c^{\infty}$. We choose $\epsilon>0$ as half the distance to the boundary of the domain $D$. In fact, we will be choosing a sequence of sets $M$ that exhaust the upper half space. The parameter $\epsilon$ will decrease as $M$ increases. 

For a fixed $l>0$ consider the following cylinder, $\Omega^1:=D(0,l)\times(0,l)$.  Let $\partial'\Omega^1=\partial D(0,l)\times (0,l)$.  Now consider the following Green's function $G^1$ adapted to $\Omega^1$:
$$
\left\{
\begin{array}{cccc}
\left(\partial_t+\Delta\right)G^1 & = & -\delta_{0,1} & \textnormal{in } \Omega^1,\\
G^1 & = & 0 & \textnormal{on } \partial'\Omega^1,\\
G^1 & = & 0 & \textnormal{when } t=l.
\end{array}
\right..
$$
with $\delta_{x,t}$ denoting the delta measure at the point $(x,t)$.  Let $k(x,t):=\f{1}{(4\pi t)^{m/2}}\exp\left(\f{-\abs{x}^2}{4t}\right)$ denote the heat kernel in $\R^{m+1}_+$.  It is important to keep in mind that $G^1(0,0)\to k(0,1)$ as $l\to\infty$. Also, one has $G(x,0)\ge a(1-||x||)$ for some positive $a$, and so the outward normal derivative of $G$ near the bottom of the cylinder is negative.

We will need to scale the cylinders in the course of the proof, so we also need a Green's function $G^R$ adapted to the cylinder $\Omega^R:=D(0,Rl)\times(0,lR^2)$, which satisfies
$$
\left\{
\begin{array}{cccc}
\left(\partial_t+\Delta\right)G^R & = & -\delta_{0,R^2} & \textnormal{in } \Omega^R,\\
G^R & = & 0 & \textnormal{on } \partial'\Omega^R,\\
G^R & = & 0 & \textnormal{when } t=lR^2.
\end{array}
\right..
$$
It is straightforward to see that we have $G^R(x,t)=\f{1}{R^m}G^1\bigl(\f{x}{R},\f{t}{R^2}\bigr)$.  

First, we estimate $b(0,R^2)$.  Using the size estimate of the function $B$ obtained from Theorem~\ref{Bellman}, we  have
$$
b(0,R^2) \leq (1+\epsilon)(p-1)(\widetilde{\norm{\varphi}_E^p}(0,R^2))^{1/p}(\widetilde{\norm{\psi}_E^q}(0,R^2))^{1/q},
$$
which gives, using the definition of the heat extension,
\begin{eqnarray*}
b(0, R^2) &\leq &\f{(1+\epsilon)(p-1)}{(4\pi R^2)^{m/2}}\left(\int_{\R^{m}}\norm{\varphi(x)}_E^pe^{-\f{\abs{x}^2}{4R^2}}\right)^{1/p}\left(\int_{\R^{m}}\norm{\psi(x)}_E^qe^{-\f{\abs{x}^2}{4R^2}}\right)^{1/q}\\
&\leq &\f{(1+\epsilon)(p-1)}{(4\pi R^2)^{m/2}}\left(\int_{\R^{m}}\norm{\varphi(x)}_E^p\right)^{1/p}\left(\int_{\R^{m}}\norm{\psi(x)}_E^q\right)^{1/q}.
\end{eqnarray*}

Now we apply Green's formula in the cylinder $\Omega_{R,\delta}:=\Omega^R\cap\{t>\delta\}$: 
\begin{eqnarray*}
b(0,R^2) & = & -\iint_{\Omega_{R,\delta}} b(x,t)\left(\partial_t+\Delta\right)G^R(x,t)dxdt\\
 & = & \iint_{\Omega_{R,\delta}}G^R(x,t)\left(\partial_t-\Delta\right)b(x,t)dxdt\\ 
 & + & \int_{D(0,R^2)}b(x,\delta)G^R(x,\delta)dx\\
 & + & \int_{\partial'\Omega^R\cap\{t>\delta\}}\left(G^R\f{\partial b}{\partial n_{outer}}-b\f{\partial G^R}{\partial n_{outer}}\right)dsdt\\
 & \geq & \iint_{\Omega^R}G^R(x,t)\left(\partial_t-\Delta\right)b(x,t)dxdt.\\  
 \end{eqnarray*}
The last inequality follows since $b$ and $G^R$ are non-negative, $G^R=0$ on the lateral boundary, and $G^R$ has a negative outer normal derivative near $t=0$. 
Thus, we have the following inequality
\begin{eqnarray*}
&&\iint_{\Omega_{R,\delta}}G_{\Omega_R}(x,t)\left(\partial_t-\Delta\right)b(x,t)dxdt\\ &\leq & 
 \f{(1+\epsilon)(p-1)}{(4\pi R^2)^{m/2}}\left(\int_{\R^{m}}\norm{\varphi(x)}_E^p\right)^{1/p}\left(\int_{\R^{m}}\norm{\psi(x)}_E^q\right)^{1/q}  . 
\end{eqnarray*}

Fix $R$ and $\delta>0$ and choose the compact set $M=\{(x,t):x\in\textnormal{clos}(D(0,lR)),\delta\leq t\leq lR^2\}$.  The function $v(x,t)$ maps  compact sets to compact sets, so $K:=v(M)$ is a compact set inside the domain $D$, as defined in Theorem \ref{Bellman}.  Finally, choose $B=B_{\epsilon, p, K}$ from Theorem~\ref{Bellman}.  We now apply the heat operator to $b(x,t)=B(v(x,t))$, which furnishes the remainder of the argument.
\begin{lm}
\label{heatcomp}
Let
$$
v(x,t):=(\widetilde{\norm{\varphi}_E^p}(x,t),\widetilde{\norm{\psi}_E^q}(x,t),\widetilde{\varphi}(x,t),\widetilde{\psi}(x,t))
$$
and $b(x,t)=B(v(x,t))$.  Let $\mathbb{E}=\R\times\R\times E\times E$.  Then, for all $(x,t)\in M$, we have
$$
\left(\partial_t-\Delta\right)b(x,t)=\sum_{i=1}^m\langle\left(-d^2B\right)\partial_i v,\partial_i v\rangle_{\mathbb{E}}.
$$
\end{lm}
\begin{proof}
This is a straightforward computation.  By the chain rule we have,
\begin{eqnarray*}
\partial_t b(x,t) & = & \langle\nabla B(v(x,t)),\partial_t v(x,t)\rangle_{\mathbb{E}}\\
\partial_i^2 b(x,t) & = & \langle\nabla B(v(x,t))\partial_i^2 v(x,t)\rangle_{\mathbb{E}}+\langle d^2B(v(x,t))\partial_i v(x,t),\partial_i v(x,t)\rangle_{\mathbb{E}}.
\end{eqnarray*}
Since the components of $v(x,t)$ are given by the heat extensions we have
\begin{eqnarray*}
\left(\partial_t-\Delta\right)b(x,t) & = & \langle\nabla B(v(x,t)),\left(\partial_t-\Delta\right)v(x,t)\rangle_{\mathbb{E}}+\sum_{i=1}^m\langle\left(-d^2B\right)\partial_i v,\partial_i v\rangle_{\mathbb{E}}\\
 & = & \sum_{i=1}^m\langle\left(-d^2B\right)\partial_i v,\partial_i v\rangle_{\mathbb{E}}.
\end{eqnarray*}
\end{proof}
We also will need the following lemma from \cite{DragVol1} that will allow us to get an $\ell^2$ sum, as opposed to an $\ell^1$ sum.
\begin{lm}
\label{ellipse}
Let $m,n,k\in\N$.  Denote $d=m+n+k$.  For $v\in\R^d$ write $v=v_m\oplus v_n\oplus v_k$, where $v_i\in\R^i$ and $i=m,n,k$.  Let $R=\norm{v_m}$ and $r=\norm{v_n}$.  Suppose a self-adjoint $d\times d$ matrix $A$ satisfies
$$
\langle Av,v\rangle\geq 2Rr,
$$
for all $v\in\R^d$.  Then there exists a $\tau>0$ satisfying
$$
\langle Av,v\rangle\geq\tau R^2+\f{1}{\tau}r^2,
$$
for all $v\in\R^d$.
\end{lm}
By Theorem~\ref{Bellman}, we have
$$
-d^2 B(\Xi,\Gamma,\xi,\gamma)\geq 2\norm{d\xi}\norm{d\gamma}.
$$
Lemma \ref{heatcomp} and Lemma \ref{ellipse} then imply that
$$
\left(\partial_t-\Delta\right)b(x,t)=\sum_{i=1}^m\langle\left(-d^2B\right)\partial_i v,\partial_i v\rangle_{\mathbb{E}}\geq\sum_{i=1}^m\left(\tau\norm{\partial_i\tilde{\varphi}}_E^2+\f{1}{\tau}\norm{\partial_i\tilde{\psi}}_E^2\right).
$$
for some $\tau>0$. Consequently, we have
$$
\left(\partial_t-\Delta\right)b(x,t)\geq 2\left(\sum_{i=1}^m\norm{\partial_i\tilde{\varphi}}_E^2\right)^{1/2}\left(\sum_{i=1}^m\norm{\partial_i\tilde{\psi}}_E^2\right)^{1/2}.
$$
Thus, using the scaling properties of the Green's functions to different-sized cylinders, we are left with the following estimate:
\begin{eqnarray*}
&& 2\iint_{\Omega_{R,\delta}}G^1\Bigl(\f{x}{R},\f{t}{R^2}\Bigr)\left(\sum_{i=1}^m\norm{\partial_i\tilde{\varphi}}_E^2\right)^{1/2}\left(\sum_{i=1}^m\norm{\partial_i\tilde{\psi}}_E^2\right)^{1/2}dxdt\\ & \leq &  \f{(1+\epsilon)(p-1)}{(4\pi)^{m/2}}\left(\int_{\R^{m}}\norm{\varphi(x)}_E^p\right)^{1/p}\left(\int_{\R^{m}}\norm{\psi(x)}_E^q\right)^{1/q}  . 
\end{eqnarray*}
Recall that $M=\{(x,t):x\in\textnormal{clos}(D(0,lR)),\delta\leq t\leq lR^2\}$.  Fix any compact subset $M_0$ of $\R^{m+1}_+$ and choose $R>0$ and $\delta>0$ such that $M_0\subset M$.  Restrict the integration in the above inequality to the set $M_0$ and let $R\to\infty$.  Then $G^1\bigl(\f{x}{R},\f{t}{R^2}\bigr)\to G^1(0,0)$.  So we are left with
\begin{eqnarray*}
&& 2G^1(0,0)\iint_{M_0}\left(\sum_{i=1}^m\norm{\partial_i\tilde{\varphi}}_E^2\right)^{1/2}\left(\sum_{i=1}^m\norm{\partial_i\tilde{\psi}}_E^2\right)^{1/2}dxdt\\ & \leq & 
 \f{(1+\epsilon)(p-1)}{(4\pi)^{m/2}}\left(\int_{\R^{m}}\norm{\varphi(x)}_E^pdx\right)^{1/p}\left(\int_{\R^{m}}\norm{\psi(x)}_E^qdx\right)^{1/q} . 
\end{eqnarray*}
Next, we let $\Omega^1=D(0,l)\times (0,l)$ fill out all of $\R^{m+1}_+$ by letting $l\to\infty$.  But as $l\to\infty$ we have that $G_{\Omega}(0,0)\to k(0,1)=\f{1}{(4\pi)^{m/2}}$ where $k(x,t)$ is the heat kernel for $\R^m$.  Since $M_0$ and $\epsilon>0$ were arbitrary, we are left with the following
$$
2\iint_{\R^{m+1}_+}\left(\sum_{i=1}^m\norm{\partial_i\tilde{\varphi}}_E^2\right)^{1/2}\left(\sum_{i=1}^m\norm{\partial_i\tilde{\psi}}_E^2\right)^{1/2}dxdt \leq (p-1)\norm{\varphi}_{L^p(\R^m;E)}\norm{\psi}_{L^q(\R^m;E)}.
$$
\end{proof}
\subsection{Construction of the Bellman function: proof of Theorem \ref{Bellman}}
In this section we construct a Bellman function from Theorem \ref{Bellman}.  The construction  is dependent upon some results of Burkholder on $L^p$ estimates for martingales.  

Let $\mathcal{D}$ denote the usual dyadic grid in $\R$, and let $\{h_I\}_{I\in\mathcal{D}}$ denote the Haar basis adapted to the dyadic grid.  For continuous compactly supported functions we have a Haar decomposition as
$$
f=\sum_{I\in\mathcal{D}}\langle f,h_I\rangle h_I.
$$
We will also need the average of a function on a dyadic interval $I$, $\langle f\rangle_I:=\f{1}{\abs{I}}\int_I f(x)dx$.

Let $\{\sigma_I\}_{I\in\mathcal{D}}$ denote a collection of complex numbers with modulus less than or equal to one.  Define the corresponding \textit{Haar multiplier} by
$$
T_\sigma(f):=\sum_{I\in\mathcal{D}}\sigma_I\langle f,h_I\rangle h_I.
$$
In the scalar case, the estimate $\sup_\sigma\norm{T_\sigma}_{p\to p}\leq p^*-1$ can be deduced from the following lemma of Burkholder \cite{Burk}.  We need a vector-valued version of the lemma and so in the corresponding Haar multiplier $T_\sigma$ we will take $\{\sigma_I\}$ as a collection of unitary operators. 
\begin{lm}
Let $E$ be a separable Hilbert space.  Let $(X_n,F_n,P)$ and $(Y_n,F_n,P)$ be $E$-valued martingales on the measure space $(\Omega,P)$.  If for almost every $\omega$ we have
$$
\norm{X_0(\omega)}_E\leq\norm{Y_0(\omega)}_E,\quad\norm{X_n(\omega)-X_{n-1}(\omega)}_E\leq\norm{Y_n(\omega)-Y_{n-1}(\omega)}_E\quad\forall n.
$$
Then
$$
\norm{X_n}_{L^p((\Omega,P);E)}\leq(p^*-1)\norm{Y_n}_{L^p((\Omega,P);E)}
$$
\end{lm}
This lemma implies the following theorem, which is the key ingredient for the construction of the Bellman function.
\begin{thm}
\label{BellmanFunct}
Let $J\in\mathcal{D}$, $E$ a separable Hilbert space, $f\in L^p(J;E)$ and $g\in L^q(J;E)$ with $\f{1}{p}+\f{1}{q}=1$.  Suppose that $p\geq 2$, then
$$
\f{1}{4\abs{J}}\sum_{I\subset J}\abs{I}\norm{\langle f\rangle_{I_+}-\langle f\rangle_{I_-}}_E\norm{\langle g\rangle_{I_+}-\langle g\rangle_{I_-}}_E\leq(p-1)\langle\norm{f}_E^p\rangle^{1/p}\langle\norm{g}_E^q\rangle^{1/q}.
$$
\end{thm}
\begin{proof}
Without loss of generality we can take $J=[0,1]$.  Let $F_n$ denote the $\sigma$-algebra generated by the dyadic intervals on $\R$ of length at least $2^{-n}$.  For $\omega\in[0,1]$ put
$$
Y_n(\omega):=\sum_{I\subset J,\abs{I}\geq 2^{-n}}\langle f, h_I\rangle h_I(\omega).
$$
Note that $\langle f, h_I\rangle$ is an element of the Hilbert space $E$.  Fix any sequence of unitary operators $\{\sigma_I\}_{I\in\mathcal{D}}$ on $E$, and define,
$$
X_n(\omega):=\sum_{I\subset J,\abs{I}\geq 2^{-n}}\sigma_I\langle f, h_I\rangle h_I(\omega).
$$
Both $(X_n,F_n,dx)$ and $(Y_n, F_n, dx)$ are $E$-valued martingales.  Clearly,
$$
Y_n-Y_{n-1}=\sum_{I\subset J,\abs{I}=2^{-n}}\langle f,h_I\rangle h_I\quad\textnormal{and}\quad X_n-X_{n-1}=\sum_{I\subset J,\abs{I}=2^{-n}}\sigma_I\langle f,h_I\rangle h_I
$$
and satisfy the hypotheses of Burkholder's Subordination Lemma.  Taking a limit as $n\to\infty$ then gives,
$$
\norm{T_\sigma(f)}_{L^p(E)}=\lim_{n\to\infty}\norm{X_n}_{L^p(E)}\leq (p-1)\lim_{n\to\infty}\norm{Y_n}_{L^p(E)}=(p-1)\norm{f}_{L^p(E)}.
$$
Duality then implies that
$$
\abs{\langle T_\sigma f,g\rangle}\leq (p-1)\norm{f}_{L^p(E)}\norm{g}_{L^q(E)}.
$$
The above inequality can then be recognized, by definition of the left-hand side and interpretation of the right-hand side, as
$$
\left\vert\f{1}{\abs{J}}\sum_{I\in\mathcal{D}}\langle \sigma_I\langle f,h_I\rangle,\langle g, h_I\rangle\rangle_E\right\vert\leq(p-1)\langle\norm{f}_E^p\rangle^{1/p}\langle\norm{g}_E^q\rangle^{1/q}.
$$
Freedom of selection of $\{\sigma_I\}_{I\in\mathcal{D}}$ and $\langle f, h_I\rangle=\f{\abs{I}^{1/2}}{2}(\langle f\rangle_{I_+}-\langle f\rangle_{I_-})$, then allows the theorem to follow.
\end{proof}

We use this theorem to construct the Bellman function.
\begin{thm}
For the domain
$$
D_p:=\{(\Xi,\Gamma,\xi,\gamma)\in \R\times\R\times E\times E: \norm{\xi}^p<\Xi,\norm{\gamma}^q<\Gamma\}
$$
there exists a function $B(\Xi,\Gamma,\xi,\gamma)$ such that for any four-tuples $a=(\Xi,\Gamma,\xi,\gamma)$, $a_-=(\Xi_-,\Gamma_-,\xi_-,\gamma_-)$, $a_+=(\Xi_+,\Gamma_+,\xi_+,\gamma_+)$ with $a=\f{1}{2}(a_-+a_+)$ we have
$$
B(a)-\f{1}{2}(B(a_+)+B(a_-))\geq\f{1}{4}\norm{\xi_--\xi_+}_E\norm{\gamma_--\gamma_+}_E.
$$
Also, everywhere in $D_p$ we have $0\leq B(a)\leq (p-1)\Xi^{1/p}\Gamma^{1/q}$.  Furthermore, let $K$ be any compact subset of $D_p$ and $\epsilon>0$.  Then there exists a function $B:=B_{\epsilon, p,K}(\Xi,\Gamma,\xi,\gamma)$ infinitely differentiable in a small neighborhood of $K$ such that
\begin{itemize}
\item[(i)] $0\leq B\leq (1+\epsilon)(p-1)\Xi^{1/p}\Gamma^{1/q}$,
\item[(ii)] $-d^2B\geq 2\norm{d\xi}\norm{d\gamma}$.
\end{itemize}
\end{thm}
\begin{proof}
Fix $(\Xi,\Gamma,\xi,\gamma)\in D_p$.  Consider all $E$-valued functions $f$ and $g$ on $J\in\mathcal{D}$ such that $\Xi=\langle\norm{f}_E^p\rangle$, $\Gamma=\langle\norm{g}_E^p\rangle$, $\xi=\langle f\rangle$, and $\gamma=\langle g\rangle$.  Set
$$
B(\Xi,\Gamma,\xi,\gamma):=\sup\left\{\f{1}{4\abs{J}}\sum_{I\subset J}\abs{I}\norm{\langle f\rangle_{I_+}-\langle f\rangle_{I_-}}_E\norm{\langle g\rangle_{I_+}-\langle g\rangle_{I_-}}_E\right\}.
$$
This supremum is independent of $J$, and allows one to show that first estimate holds for the function $B$.  The inequality $0\leq B(a)\leq (p-1)\Xi^{1/p}\Gamma^{1/q}$ follows from Theorem \ref{BellmanFunct}.  The second half of the proof is achieved by a mollification of the function $B$, this is similar to what appears in \cite{VolNaz}.
\end{proof}

We have thus finished the proof of Theorem \ref{Bellman}.
\section{Conclusion}
\label{s5}
In this paper, we have obtained a new estimate of the type
\eq[e5]{
\norm{\mathcal{S}}_{L^p(\R^n;\Lambda)\to L^p(\R^n;\Lambda)}\le C(n)(p^*-1).
}
The previous best constant was $C(n)=\frac43n-2$ and ours is $C(n)=n.$ Although there is still linear growth in $n$ and the result is far from
the Iwaniec--Martin conjecture \eqref{e1}, we believe our approach holds unexhausted potential. More specifically, we expect to be able to reduce the order of the growth from $n$ to, approximately, $\sqrt n$, asymptotically as $p\to\infty$.

Additionally, the approach used in this paper can be used to study the corresponding weighted questions for the operator $\mathcal{S}$.  In the case of the complex plane, see \cite{PetermichlVol}, the sharp dependence on the $A_p$ characteristic was important to establish the minimal regularity of solutions to the Beltrami equation.  There are corresponding weighted questions for the operator $\mathcal{S}$ that should be explored, and the tools established in this paper will be invaluable in that endeavor.

In \cite{DragVol2}, the authors studied the planar Beurling--Ahlfors transform $\mathcal{T}.$ They employed the now traditional ``Heat extension--Green's function--Bellman function'' chain, together with subtle estimates using the rotational structure of the operator. Recall that $\mathcal{T}=R^2_1-R_2^2+2iR_1R_2$ and the part $2R_1R_2$ is just a rotation of the part $R^2_1-R^2_2.$ Estimating the two separately yields the estimate 
\eq[e6]{
\norm{\mathcal{T}}_{p\to p}\le 2(p^*-1).
}
We believe that the $2$ in \eqref{e6} is the equivalent of our $n,$ since $\mathcal{T}$ is just $\mathcal{S}$ acting on $1$-forms on $\mathbb{R}^2.$ Volberg and  Dragi{\v{c}}evi{\'c} estimated the modified operator
$$
\mathcal{T}_\theta=(R^2_1-R^2_2)\cos\theta+2R_1R_2\sin\theta)
$$
and were able to reduce the $2$ in \eqref{e6} to approximately $\sqrt2,$ for large $p.$ One can then interpolate between these two estimates to produce a significant reduction of the estimate for all $p.$ 

We intend to concentrate on the action of the generalized operator $\mathcal{S}$ on forms in even dimensions, an application particularly important for quasiregular mappings~(\cite{IwanMart1}). In this setting, it seems very plausible that 
$$
\norm{\mathcal{S}}_{L^p(\R^{2m};\Lambda)\to L^p(\R^{2m};\Lambda)}=\norm{\mathcal{S}}_{L^p(\R^{2m};\Lambda_m)\to L^p(\R^{2m};\Lambda_m)},
$$
due to the block structure of operator $\mathcal{S}$ and the obvious symmetry. One then observes that the ``middle'' block $\mathcal{S}_m$ exhibits a rotation-like structure, very similar to that of $\mathcal{T}.$ It is this structure that we intend to exploit, hoping, therefore, to obtain
$$
C(n)\approx\sqrt{n}
$$
for large $p.$ If that is successful, we can again interpolate obtaining a much improved estimate for all $p.$

\section*{References}
\bibliographystyle{plain}

\begin{biblist}

\bib{BanPrabhu}{article}{
  author={Ba{\~n}uelos, Rodrigo},
   author={Janakiraman, Prabhu},
   title={$L^p$-bounds for the Beurling-Ahlfors transform},
   journal={Tras. Amer. Math. Soc.},
   date={to appear}
}

\bib{BanLind}{article}{
   author={Ba{\~n}uelos, Rodrigo},
   author={Lindeman, Arthur, II},
   title={A martingale study of the Beurling--Ahlfors transform in ${\bf
   R}\sp n$},
   journal={J. Funct. Anal.},
   volume={145},
   date={1997},
   number={1},
   pages={224--265},
}

\bib{BanMen}{article}{
   author={Ba{\~n}uelos, Rodrigo},
   author={M{\'e}ndez-Hern{\'a}ndez, Pedro J.},
   title={Sharp inequalities for heat kernels of Schr\"odinger operators and
   applications to spectral gaps},
   journal={J. Funct. Anal.},
   volume={176},
   date={2000},
   number={2},
   pages={368--399}
}

\bib{BanWang}{article}{
   author={Ba{\~n}uelos, Rodrigo},
   author={Wang, Gang},
   title={Sharp inequalities for martingales with applications to the
   Beurling-Ahlfors and Riesz transforms},
   journal={Duke Math. J.},
   volume={80},
   date={1995},
   number={3},
   pages={575--600}
}
\bib{Burk}{incollection}{
author={Burkholder, Donald L.},
title={Explorations in martingale theory and its applications},
booktitle={Ecole d'Et\'{e} de Probabilit\'{e}s de Saint-Flour XIX –- 1989, Lecture Notes in Math., vol. 1464, Springer. Berlin},
publisher={Springer, Berlin},
year={1991},
pages={1--66}
}


\bib{DragTreilVol}{article}{
   author={Dragi{\v{c}}evi{\'c}, Oliver},
   author={Treil, Sergei},
   author={Volberg, Alexander},
   title={A Theorem about Three Quadratic Forms},
   journal={preprint at arXiv:0710.3249}
}

\bib{DragVol1}{article}{
   author={Dragi{\v{c}}evi{\'c}, Oliver},
   author={Volberg, Alexander},
   title={Bellman functions and dimensionless estimates of Littlewood--Paley
   type},
   journal={J. Operator Theory},
   volume={56},
   date={2006},
   number={1},
   pages={167--198},
}
				
\bib{DragVol2}{article}{
   author={Dragi{\v{c}}evi{\'c}, Oliver},
   author={Volberg, Alexander},
   title={Bellman function, Littlewood-Paley estimates and asymptotics for
   the Ahlfors-Beurling operator in $L\sp p(\mathbb{C})$},
   journal={Indiana Univ. Math. J.},
   volume={54},
   date={2005},
   number={4},
   pages={971--995}
}

\bib{DragVol3}{article}{
   author={Dragi{\v{c}}evi{\'c}, Oliver},
   author={Volberg, Alexander},
   title={Bellman function for the estimates of Littlewood-Paley type and
   asymptotic estimates in the $p-1$ problem},
   language={English, with English and French summaries},
   journal={C. R. Math. Acad. Sci. Paris},
   volume={340},
   date={2005},
   number={10},
   pages={731--734}
}

\bib{FeffKenPip}{article}{
   author={Fefferman, Robert A.},
   author={Kenig, Carlos E.},
   author={Pipher, Jill C.},
   title={The theory of weights and the Dirichlet problem for elliptic
   equations},
   journal={Ann. of Math. (2)},
   volume={134},
   date={1991},
   number={1},
   pages={65--124}
}

\bib{Iwan}{article}{
author={Iwaniec, Tadeusz},
title={Extremal inequalities in Sobolev spaces and quasiconformal mappings},
journal={Z. Anal. Anwendungen},
volume={1},
date={1982},
number={6},
pages={1--16}
}

\bib{IwanMart1}{book}{
   author={Iwaniec, Tadeusz},
   author={Martin, Gaven},
   title={Geometric function theory and non-linear analysis},
   series={Oxford Mathematical Monographs},
   publisher={The Clarendon Press Oxford University Press},
   place={New York},
   date={2001},
   pages={xvi+552},
   isbn={0-19-850929-4}
}



\bib{IwanMart2}{article}{
   author={Iwaniec, Tadeusz},
   author={Martin, Gaven},
   title={Riesz transforms and related singular integrals},
   journal={J. Reine Angew. Math.},
   volume={473},
   date={1996},
   pages={25--57}
}
		
\bib{IwanMart3}{article}{
   author={Iwaniec, Tadeusz},
   author={Martin, Gaven},
   title={Quasiregular mappings in even dimensions},
   journal={Acta Math.},
   volume={170},
   date={1993},
   number={1},
   pages={29--81},
}

\bib{VolNaz}{article}{
   author={Nazarov, Fedor},
   author={Volberg, Alexander},
   title={Heat extension of the Beurling operator and estimates for its
   norm},
   language={Russian, with Russian summary},
   journal={Algebra i Analiz},
   volume={15},
   date={2003},
   number={4},
   pages={142--158},
   translation={
      journal={St. Petersburg Math. J.},
      volume={15},
      date={2004},
      number={4},
      pages={563--573},
   }
}


\bib{PetermichlVol}{article}{
   author={Petermichl, Stefanie},
   author={Volberg, Alexander},
   title={Heating of the Ahlfors-Beurling operator: weakly quasiregular maps
   on the plane are quasiregular},
   journal={Duke Math. J.},
   volume={112},
   date={2002},
   number={2},
   pages={281--305}
}


\end{biblist}

\end{document}